\documentclass[12pt]{amsart}%{tran-l}%{amsart}

\author{Paul \textsc{Poncet}}
\address{CMAP, \'{E}cole Polytechnique, Route de Saclay, 91128 Palaiseau Cedex, France} 

\email{poncet@cmap.polytechnique.fr}

\usepackage[colorinlistoftodos,bordercolor=orange,backgroundcolor=orange!20,linecolor=orange,textsize=scriptsize]{todonotes}
\usepackage{stmaryrd}
\usepackage{amssymb}
\usepackage{amsmath}
\usepackage{yhmath}
\usepackage{calrsfs} % nouvelles lettres calligraphiques (commande \mathrsfs)

\usepackage{times} % sans que je comprenne pourquoi, les recherches textuelles dans les documents pdf fonctionnent avec cette fonte
\usepackage{bbm}
\usepackage{ifpdf}

\def\twoheaddownarrow{\rlap{$\downarrow$}\raise-.5ex\hbox{$\downarrow$}}%%
\def\twoheaduparrow{\rlap{$\uparrow$}\raise.5ex\hbox{$\uparrow$}}%%

\newtheorem*{theorem*}{Theorem}

\newtheorem{theorem}{Theorem}[section]
\newtheorem{corollary}[theorem]{Corollary}
\newtheorem{proposition}[theorem]{Proposition}

\theoremstyle{definition}

\begin{document}

\title{Galois connections between closure spaces}

\date{\today}

 %Galois correspondences, closure operators (in relation to ordered sets)


\begin{abstract}
We extend Galois connections between posets to Galois connections between closure spaces, and prove a characterization theorem.
\end{abstract}

\maketitle
%\tableofcontents

%%%%%%%%%%%%%%%%%%%%%%
%%%%%%%%%%%%%%%%%%%%%%
%%%%%%%%%%%%%%%%%%%%%%
%%%%%%%%%%%%%%%%%%%%%%
\section{Introduction}

%\todo{ Quelles considerations fonctorielles ? }

%\todo{ Peut-on encore davantage generaliser ? preclosure spaces? }

Galois connections were introduced by Ore \cite{Ore44} and have proved useful in a wide variety of mathematical areas. 
While Galois connections play on the ground of posets (or more generally of quasiordered sets or \textit{qosets}), we extend this notion to that of closure spaces.

%%%%%%%%%%%%%%%%%%%%%%
%%%%%%%%%%%%%%%%%%%%%%
%%%%%%%%%%%%%%%%%%%%%%
%%%%%%%%%%%%%%%%%%%%%%
\section{Qosets, Galois connections, and closure spaces}

A \textit{quasiordered set} or \textit{qoset} $(P,\leqslant)$ is a set $P$ together with a reflexive and transitive binary relation $\leqslant$. 
If in addition $\leqslant$ is antisymmetric, then $(P, \leqslant)$ is a \textit{partially ordered set} or \textit{poset}. 

A \textit{Galois connection between qosets} $P$ and $P'$ is a pair $(\varphi, \psi)$ of maps $\varphi : P \to P'$ and $\psi : P' \to P$ such that $\varphi(x) \leqslant x'$ if and only if $x \leqslant \psi(x')$ for all $x \in P$, $x' \in P'$. 
The map $\varphi$ (resp.\ $\psi$) is the \textit{left adjoint} (resp.\ \textit{right adjoint}) of the Galois connection. 
We refer the reader to Ern\'{e} et al.\ \cite{Erne93} for usual properties, multiple examples, and various references on Galois connections. 
Note the following properties: 
\begin{itemize}
  \item $\varphi$ and $\psi$ are order-preserving;
  \item $x \leqslant \psi(\varphi(x))$ for all $x \in P$;
  \item $x' \geqslant \varphi(\psi(x'))$ for all $x' \in P'$;
  \item $x = \psi(\varphi(x))$ for all $x \in P$ iff $\varphi$ is injective iff $\psi$ is surjective;
  \item $x' = \varphi(\psi(x'))$ for all $x' \in P'$ iff $\psi$ is injective iff $\varphi$ is surjective. 
\end{itemize} 

A \textit{closure space} $(E, [\cdot])$ is a set $E$ equipped with a map $[\cdot] : 2^E \to 2^E$ such that ${ A } \subseteq { [[A]] } \subseteq { [B] }$, for all subsets $A$, $B$ of $E$ such that ${ A } \subseteq { B }$. 
The map $[\cdot]$ is called a \textit{closure operator}. 
A subset $F$ of $E$ is \textit{$[\cdot]$-closed} (or simply \textit{closed} if the context is clear) if $[F] = F$. 
%A subset $G$ of $E$ is \textit{$[\cdot]$-open} (or simply \textit{open} if the context is clear) if $E \setminus G$ is closed. 
Note that the empty set $\emptyset$ is not assumed to be closed in general. 
Recall that the set of closed subsets of a closure space is a \textit{Moore family}, in the sense that it is stable under arbitrary intersections (hence is a complete lattice). 

A map $f : E \to E'$ between two closure spaces $E$ and $E'$ is \textit{continuous} if ${ f([A]) } \subseteq { [f(A)] }$, for all subsets $A$ of $E$. 
An equivalent condition is that $f^{-1}(F')$ be closed in $E$ for every closed subset $F'$ of $E'$. 

Every closure space $(E, [\cdot])$ induces a qoset $(E, \leqslant)$, where the quasiorder $\leqslant$ is defined by $x \leqslant y$ if $x \in [y]$, where we write $[y]$ instead of $[\{ y \}]$. 
The quasiorder $\leqslant$ is called the \textit{specialization order} on $(E, [\cdot])$. 
Every qoset $(P, \leqslant)$ induces a closure space $(P, \downarrow\!\! \cdot)$, where the closure operator $\downarrow\!\! \cdot$ is defined by $\downarrow\!\! A := \{ x \in P : \exists a \in A, x \leqslant a \}$. 
The subset $\downarrow\!\! A$ is called the \textit{lower set} generated by $A$, and the closure operator $\downarrow\!\! \cdot$ is called the \textit{Alexandrov closure operator}. 
 
We refer the reader to Ern\'e \cite{Erne09} for more on closure spaces. 

%%%%%%%%%%%%%%%%%%%%%%
%%%%%%%%%%%%%%%%%%%%%%
%%%%%%%%%%%%%%%%%%%%%%
%%%%%%%%%%%%%%%%%%%%%%
\section{Galois connections between closure spaces}

We define a \textit{Galois connection between closure spaces} $E$ and $E'$ as a pair $(\varphi, \psi)$ of maps $\varphi : E \to E'$ and $\psi : E' \to E$ such that
\[
{ \varphi^{-1}([A']) } = { [\psi(A')] },
\]
for all subsets $A'$ of $E'$. 

The following result shows that the previous definition indeed generalizes the usual notion of Galois connection.

\begin{proposition}\label{prop:gene}
If $(\varphi, \psi)$ is a Galois connection between qosets $P$ and $P'$, then $(\varphi, \psi)$ is a Galois connection between the closure spaces $P$ and $P'$ equipped with their respective Alexandrov closure operators. 
\end{proposition}

\begin{proof}
Let $x \in P$. 
Then 
\begin{align*}
{ x \in { \varphi^{-1}(\downarrow\!\! A') } } &\Leftrightarrow { \varphi(x) \in { \downarrow\!\! A' } } \Leftrightarrow { \exists a' \in A' : \varphi(x) \leqslant a' } \\
&\Leftrightarrow { \exists a' \in A : x \leqslant \psi(a') } \Leftrightarrow { x \in { \downarrow\!\! \psi(A') } }.
\end{align*} 
This shows that ${ \varphi^{-1}(\downarrow\!\! A') } = { \downarrow\!\! \psi(A') }$. 
\end{proof}

The next result gives a converse statement to Proposition~\ref{prop:gene} as a special case.

\begin{proposition}
If $(\varphi, \psi)$ is a Galois connection between closure spaces $E$ and $E'$, then $(\varphi, \psi)$ is a Galois connection between the qosets $E$ and $E'$ equipped with their respective specialization orders. 
\end{proposition}

\begin{proof}
Let $x \in E$, $x' \in E'$. 
Then 
\begin{align*}
{ \varphi(x) \leqslant x' } &\Leftrightarrow { \varphi(x) \in { [x'] } } \Leftrightarrow { x \in { \varphi^{-1}([x']) } } \\
&\Leftrightarrow { x \in [\psi(x')] } \Leftrightarrow { x \leqslant \psi(x') }.
\end{align*} 
So $(\varphi, \psi)$ is a Galois connection between the qosets $E$ and $E'$. 
\end{proof}

We now come to our characterization theorem.
For a closure space $E$, we write $\mathrsfs{F}(E)$ for the poset made of its closed subsets, ordered by inclusion. 

\begin{theorem}\label{thm:main}
Let $E$, $E'$ be closure spaces, and let $\varphi : E \to E'$, $\psi : E' \to E$. 
Then the following are equivalent:
\begin{enumerate}

  \item\label{thm:main1} $(\varphi, \psi)$ is a Galois connection between the closure spaces $E$ and $E'$;
  \item\label{thm:main2} $\varphi$ and $\psi$ are continuous and $(\varphi, \psi)$ is a Galois connection between the qosets $E$ and $E'$;
  \item\label{thm:main3} $\varphi$ and $\psi$ are continuous and $(\varphi^{-1}, \psi^{-1})$ is a Galois connection between the posets $\mathrsfs{F}(E')$ and $\mathrsfs{F}(E)$.
\end{enumerate}
\end{theorem}

\begin{proof}
\eqref{thm:main1} $\Rightarrow$ \eqref{thm:main2}: 
We already know from the previous proposition that $(\varphi, \psi)$ is a Galois connection between the qosets $E$ and $E'$. 
Let us show that $\varphi$ and $\psi$ are continuous. 
The continuity of $\varphi$ is clear, since ${ \varphi^{-1}(F') } = { [\psi(F')] }$ is closed in $E$ for every closed subset $F'$ of $E'$. 
To prove the continuity of $\psi$, let $x \in \psi([A'])$. 
Then $x = \psi(x')$ for some $x' \in [A']$. 
This implies that $\varphi(x) = \varphi(\psi(x')) \leqslant x'$, so that $\varphi(x) \in { [x'] } \subseteq { [A'] }$. 
So $x \in { \varphi^{-1}([A']) } = { [\psi(A')] }$. 
We have proved that ${ \psi([A']) } \subseteq { [\psi(A')] }$, for all ${ A' } \subseteq { E' }$, so $\psi$ is continuous. 

\eqref{thm:main2} $\Rightarrow$ \eqref{thm:main3}: 
Let $F \in \mathrsfs{F}(E)$, $F' \in \mathrsfs{F}(E')$. 
Suppose first that ${ \varphi^{-1}(F') } \subseteq { F }$, and let $x' \in F'$. 
Since $\varphi(\psi(x')) \leqslant x'$, we have $\varphi(\psi(x')) \in F'$.
Thus, $\psi(x') \in { \varphi^{-1}(F') } \subseteq { F }$, so $x' \in \psi^{-1}(F)$. 
This shows that ${ F' } \subseteq { \psi^{-1}(F) }$.
Conversely, suppose that ${ F' } \subseteq { \psi^{-1}(F) }$, and let $x \in \varphi^{-1}(F')$. 
Then $\varphi(x) \in { F' } \subseteq { \psi^{-1}(F) }$, so $x \leqslant \psi(\varphi(x)) \in F$. 
This shows that ${ \varphi^{-1}(F') } \subseteq { F }$. 
We have proved that ${ \varphi^{-1}(F') } \subseteq { F }$ if and only if ${ F' } \subseteq { \psi^{-1}(F) }$, for all $F \in \mathrsfs{F}(E)$, $F' \in \mathrsfs{F}(E')$, so $(\varphi^{-1}, \psi^{-1})$ is indeed a Galois connection between the posets $\mathrsfs{F}(E')$ and $\mathrsfs{F}(E)$.

\eqref{thm:main3} $\Rightarrow$ \eqref{thm:main1}: 
If $F \in \mathrsfs{F}(E)$, then 
${ \varphi^{-1}([A']) } \subseteq { F } \Leftrightarrow { [A'] } \subseteq { \psi^{-1}(F) } \Leftrightarrow { A' } \subseteq { \psi^{-1}(F) } \Leftrightarrow { \psi(A') } \subseteq { F } \Leftrightarrow { [\psi(A')] } \subseteq { F }$, where these equivalences use the hypotheses that $\varphi$ and $\psi$ are continuous and $(\varphi^{-1}, \psi^{-1})$ is a Galois connection between $\mathrsfs{F}(E')$ and $\mathrsfs{F}(E)$. 
This shows that ${ \varphi^{-1}([A']) } = { [\psi(A')] }$. 
\end{proof}

\begin{corollary}
If $E$ and $E'$ are $T_1$ topological spaces, equipped with their respective topological closure operators, then $(\varphi, \psi)$ is a Galois connection between $E$ and $E'$ if and only if $\varphi$ and $\psi$ are homeomorphisms inverse of each other.
\end{corollary}

\bibliographystyle{plain}

\begin{thebibliography}{1}

\bibitem{Erne09}
Marcel Ern{\'e}.
\newblock Closure.
\newblock In {\em Beyond topology}, volume 486 of {\em Contemp. Math.}, pages
  163--238. Amer. Math. Soc., Providence, RI, 2009.

\bibitem{Erne93}
Marcel Ern{\'e}, J\"urgen Koslowski, Austin Melton, and George~E. Strecker.
\newblock A primer on {G}alois connections.
\newblock In {\em Papers on general topology and applications (Madison, WI,
  1991)}, volume 704, pages 103--125, 1993.

\bibitem{Ore44}
Oystein Ore.
\newblock Galois connexions.
\newblock {\em Trans. Amer. Math. Soc.}, 55:493--513, 1944.

\end{thebibliography}

\def\cprime{$'$} \def\cprime{$'$} \def\cprime{$'$} \def\cprime{$'$}
  \def\ocirc#1{\ifmmode\setbox0=\hbox{$#1$}\dimen0=\ht0 \advance\dimen0
  by1pt\rlap{\hbox to\wd0{\hss\raise\dimen0
  \hbox{\hskip.2em$\scriptscriptstyle\circ$}\hss}}#1\else {\accent"17 #1}\fi}
  \def\ocirc#1{\ifmmode\setbox0=\hbox{$#1$}\dimen0=\ht0 \advance\dimen0
  by1pt\rlap{\hbox to\wd0{\hss\raise\dimen0
  \hbox{\hskip.2em$\scriptscriptstyle\circ$}\hss}}#1\else {\accent"17 #1}\fi}

\end{document}